\documentclass[12pt]{amsart}
\usepackage{amssymb}
\usepackage{graphicx,xspace}

\numberwithin{equation}{section}
\newtheorem{theorem}{Theorem}[section]

\newtheorem{defi}[theorem]{Definition}
\newtheorem{exam}[theorem]{Example}
\newtheorem{rema}[theorem]{Remark}

\author[S. Forcey]{Stefan Forcey} \address[S. Forcey]{
    Department of Mathematics\\
    The University of Akron\\
    Akron, OH 44325-4002
    }
    \email{sf34@uakron.edu}  \urladdr{http://www.math.uakron.edu/\~{}sf34/}

\author[M. Kafashan]{Mohammadmehdi Kafashan} \address[M. Kafashan]{
    Department of Electrical Engineering\\
    The University of Akron\\
    Akron, OH 44325-4002
    }

\author[M. Maleki]{Mehdi Maleki} \address[M. Maleki]{
    Department of Electrical Engineering\\
    The University of Akron\\
    Akron, OH 44325-4002
    }

\author[M. Strayer]{Michael Strayer} \address[M. Strayer]{
    Department of Mathematics\\
    The University of Akron\\
    Akron, OH 44325-4002
    }

\title{Recursive bijections for Catalan objects.}

\begin{document}

\begin{abstract}
 In this note we introduce several instructive examples of bijections found between
several different combinatorially defined sequences of sets. Each
sequence has cardinalities given by the Catalan numbers. Our results
answer
 some questions posed by R. Stanley in the addendum to his textbook. We actually discuss two
types of bijection, one defined recursively and the other defined in
a more local, relative, fashion. It is interesting to compare the
results of the two.
\end{abstract}

\maketitle
\section{Introduction}

\subsection{Catalan objects}  In the work that led to this note we set out
to find explicit bijections between several sequences of sets that are known to
be counted by the Catalan numbers, sequence A000108 in \cite{sloane}. One
sequence of sets we call the \emph{right-swept planar unary-binary trees}, or
\emph{right-swept trees} for short. These are the same restriction of planar
unary-binary trees that are labeled as example ``$www$'' in R. Stanley's
Catalan Addendum (version of July 2012) \cite{Stan}. They are also described in
\cite{kim} as a special kind of planar unary-binary trees, and there is given
in that article a bijection to the non-crossing partitions. We were inspired to
find bijections from these right-swept trees to other familiar sets of objects
counted by the Catalan numbers, due to the fact that they have a nice recursive
description that is different from the standard Catalan recursion. In this
paper we find bijections from the right-swept trees to staircase tilings,
planar trees, planar binary trees and arc tree diagrams, allowing the reader to
construct many more implied bijections to non-crossing partitions, polygonal
dissections and lattice paths. Our first set of recursive bijections is
described in Section~\ref{sec:rec}. Our second bijection between staircase
tilings and right-swept trees is discussed in Section~\ref{sec:rel}.

A \emph{right-swept tree} is a rooted planar tree with the following
restrictions. In general a node may be a leaf, may have a single child which
must be left, middle, or right; or instead may have two children: left and
right. Any left child has further restrictions: it may not be a leaf, and it
may not have a middle child. Thus any branching to the left is eventually swept
right before it can end in a leaf. Figure~\ref{rightswept} shows a right-swept
tree.

\begin{figure}[!ht] 
   \centering
   \includegraphics[width=\textwidth]{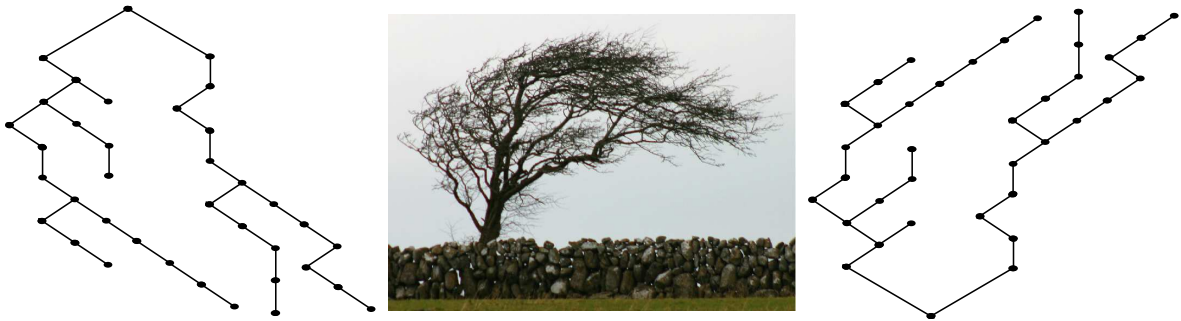}
   \caption{Two views of a right-swept tree. We will use the first, with root at the top. The
    (remixed) photo is a windswept hawthorne near Galway, original taken by Eoin Gardiner (creative commons).}
   \label{rightswept}
\end{figure}

Our second featured sequence is known as the \emph{diagonal rectangular tilings
of staircase shapes}, or \emph{staircase tilings} for short. A staircase shape
is the outline of a Young diagram corresponding to a partition given by
$(n,n-1,\dots,1).$  The Catalan numbers count tilings whose rectangles each
include some of the stepped diagonal--i.e. each intersects the end of a row in
the Young diagram. These are equivalently described as rectangular tilings of
height $n$ staircase shapes that contain exactly $n$ rectangles. The fact that
having $n$ rectangles is equivalent to being a diagonal tiling is also true for
diagonal rectangulations of the square, and we refer the reader to
\cite{reading-direc} both for a proof and for some very nice related
combinatorics. The staircase tilings are also referred to as tilings of
stair-step shapes, as in \cite{mernik}.

There is a well known bijection from staircase tilings to the sets of rooted
planar binary trees. Simply removing the ``steps,'' the vertical and horizontal
boundary segments of unit length at the far right and bottom of the figure, and
adding a root, yields a binary tree (whose drawing has been rotated from its
normal presentation.)  A staircase tiling and its corresponding binary tree is
shown in Figure~\ref{stairs}.

\begin{figure}[!ht] 
   \centering
   \includegraphics[width=\textwidth]{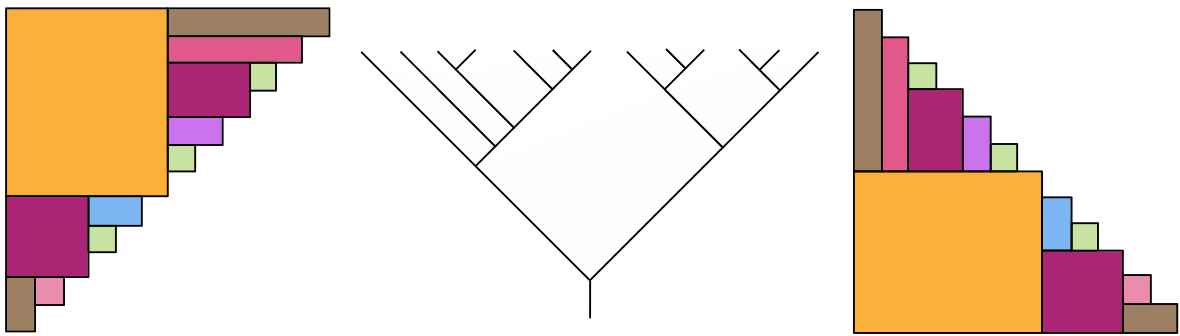}
   \caption{Two views of a staircase tiling. We will use the version on the left.
   The
    binary tree in the center is the image of the staircase tiling under the classical bijection: it is formed by removing the steps.
    We will draw rooted binary trees with the root at the bottom.}
   \label{stairs}
\end{figure}

The bijection exemplified in Figure~\ref{stairs} is trivially described in
recursive terms. A planar binary tree $t$ with more than one leaf (and thus the
corresponding staircase tiling) is formed by joining a pair of smaller binary
trees--the left and right subtrees whose root is the first branch point of $t.$
Many other Catalan objects have a similar recursive description--triangular
dissections of a polygon, bracketings  of a string of symbols, and Dyck paths,
to name a few. This description leads to Segner's classic recursion relation
for the Catalan numbers $c_n$:
\begin{equation}
{c_0} = 1\,\,\,and\,\,{c_{n + 1}} = \sum\limits_{k = 0}^n {{c_k}{c_{n - k}}}
\,\,\,\,\,for\,\,\,n > 0,  \label{Segner}
\end{equation}
...where $n$ is the number of branch points for the binary tree, or the number
of rectangles in the staircase tiling. The recursion yields the closed formula:
$$
c_n = \frac{1}{n+1}{2n \choose n}.
$$

If $X_n$ and $ X'_n$ are any two of the sequences of sets that have Segner's
recursive description then they are in piecewise bijection (both counted by
$c_n$), and the correspondence is explicitly described using the recursion. If
a bijection is given between the $k^{th}$ sets of the two sequences, for
$k=1\dots n,$ then given an object of $X_{n+1}$ we can decompose it into two
objects from earlier in the sequence, find their corresponding objects and use
them to construct the corresponding object in $X'_{n+1}.$

We began by seeking similar explicit bijections between the staircase tilings
and the right-swept trees. We found two such bijections. The first bijection
discussed in Section~\ref{sec:rec} is based on an alternate recursive
description of the staircase tilings, which fits well with the natural
recursive description of the right-swept trees.  By finding analogous ways to
recursively construct other sorts of Catalan objects we can describe them as
being in bijection with the right-swept trees, and each other, in new ways. As
an example we include \emph{non-crossing arc diagrams with distinct left
endpoints}, or \emph{arc trees} for short.

 \emph{Arc trees} are defined to be the ways of connecting $n+1$
points lying on a horizontal line on the plane with $n$ non-crossing arcs lying
above the line such that the left endpoints of the arcs are distinct. There is
always a unique series of arcs traveled from left to right from any point to
the rightmost point. Thus there is always a unique shortest path to travel from
one point to another. These are easily seen to be in bijection with planar
rooted trees with $n$ edges, simply by choosing the rightmost point to be the
root and then straightening the arcs. See Figure \ref{tree_color}.

\begin{figure}[!ht] 
   \centering
   \includegraphics[width=\textwidth]{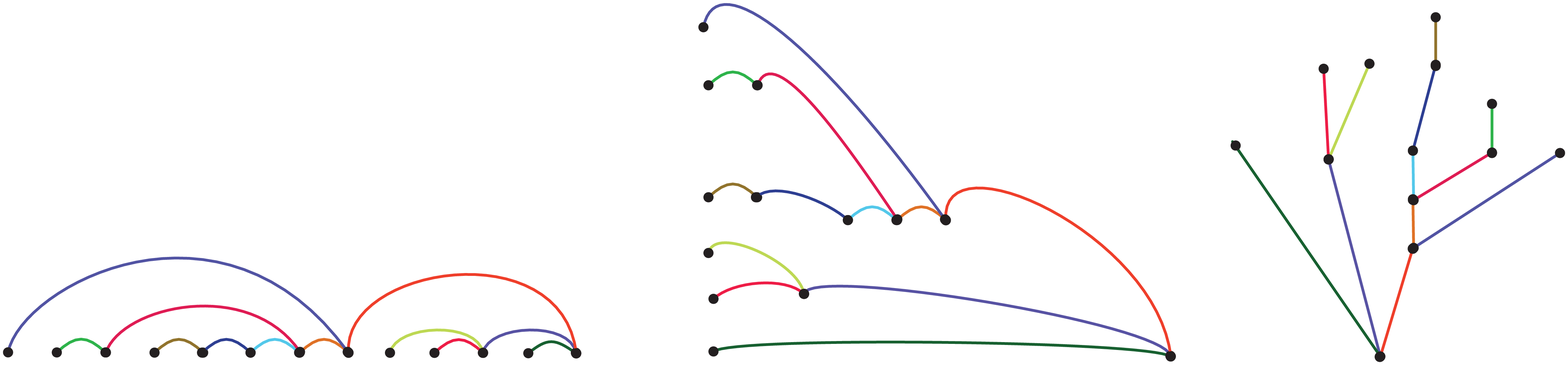}
   \caption{The bijection from non-crossing arc diagrams with distinct left endpoints to planar rooted trees:
    from left to right we gradually straighten the arcs. We will draw planar rooted trees with the root at the bottom
     (as opposed to the right-swept trees with their root at the top.)}
   \label{tree_color}
\end{figure}

\section{Recursive bijections}\label{sec:rec}

As mentioned, in order to keep this paper self contained, we have repeated the
definitions given in R. Stanley's Catalan addendum (version of 13 July 2012) to
\cite{Stan} of the combinatorial objects $www$, $h^8$, and $h^5$: called here
respectively the right-swept trees, staircase tilings and arc trees.

We represent the set of right-swept trees with $ n $ nodes as $
{{\bf{T}}_n} $, the set of stair-case tilings with $n$ rectangles as
$ {{\bf{S}}_n} $ and the set of arc trees with $n$ arcs as $
{{\bf{A}}_n}.$ We refer to the sets as the shapes of size $n.$ The
five objects for size $n=3$ are seen in Figures~\ref{www},~\ref{e^8}
and ~\ref{f^5}.

\begin{figure}[!ht] 
   \centering
   \includegraphics[scale=.28]{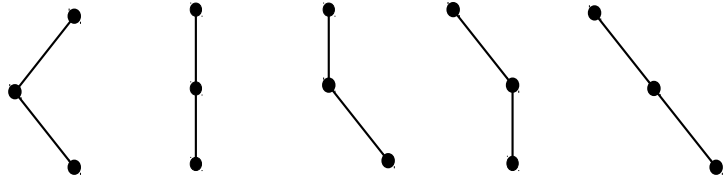}
   \caption{Right-swept trees ${{\bf{T}}_3}.$}
   \label{www}
\end{figure}

\begin{figure}[!ht] 
   \centering
   \includegraphics[scale=.22]{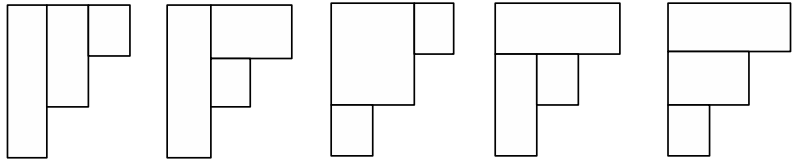}
   \caption{Staircase tilings $ {{\bf{S}}_3} $. }
   \label{e^8}
\end{figure}

\begin{figure}[!ht] 
   \centering
   \includegraphics[width=\textwidth]{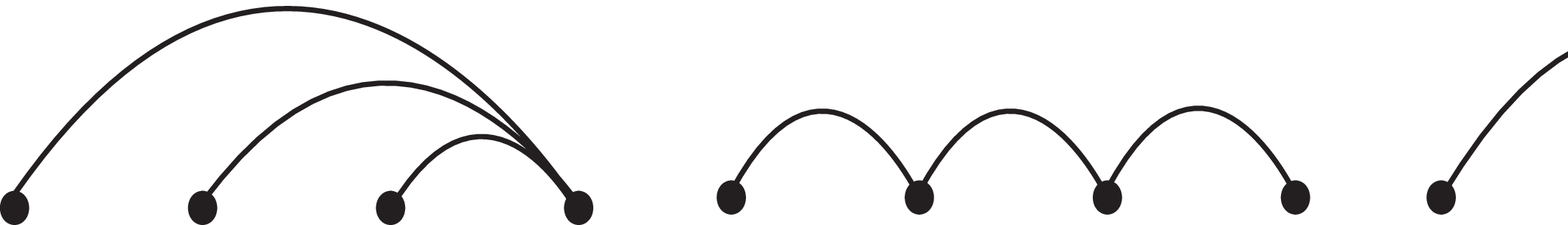}
   \caption{(Non crossing) arc trees $ {{\bf{A}}_3}. $}
   \label{f^5}
\end{figure}

 Here we introduce a
recursive method to construct a shape with size of $ n+1 $ in any of
these three combinatorial sets using shapes of smaller size. We
construct four different types of shapes of size $ n+1 $ using four
methods:
\\
\begin{enumerate}
\item We define  functions
${f_R}:{{\bf{X}}_n} \to {{\bf{X}}_{n + 1}},\,\,{\bf{X}} \in \left\{
{{\bf{T}},{\bf{S}},{\bf{A}}} \right\},\,\,n \ge 1 \in  \mathbb{N} $. Depending
on which combinatorial object is the input to this function, we perform the
following procedures:\\

a) $ {\bf{X}} = {\bf{T}} $: In this case, the output in ${{\bf{T}}_{n + 1}} =
{f_R}\left( {{{\bf{T}}_n}} \right)$ is a right-swept tree with $ n+1 $ vertices
constructed by adding one vertex to $t\in {{\bf{T}}_n}$ as the new root whose
right child is the
root of $t$.\\

b) $ {\bf{X}} = {\bf{S}} $: In this case, the output in $ {{\bf{S}}_{n + 1}} =
{f_R}\left( {{{\bf{S}}_n}} \right) $ is a staircase tiling with $ n+1 $
rectangles constructed by adding one $\left( {(n + 1) \times 1} \right)$
rectangle to the
left side of  an input from ${{\bf{S}}_n}$.\\

c) $ {\bf{X}} = {\bf{A}} $: In this case, the output in ${{\bf{A}}_{n + 1}} =
{f_R}\left( {{{\bf{A}}_n}} \right)$ is is an arc tree with $ n+2 $ points
constructed by adding one point to the left side of $a\in{{\bf{A}}_n}$ and
connecting it to the
nearest point in $a$.\\

Figure~\ref{rightchild}  represents the operation of ${f_R}$ whose input can be
any possible shape with size of $ n $. Thus the number of shapes of size $ n+1
$ which can be constructed by ${f_R}$ is denoted ${c_{n + 1,1}} = {c_n}$.

\begin{figure}[!ht] 
   \centering
   \includegraphics[scale=.5]{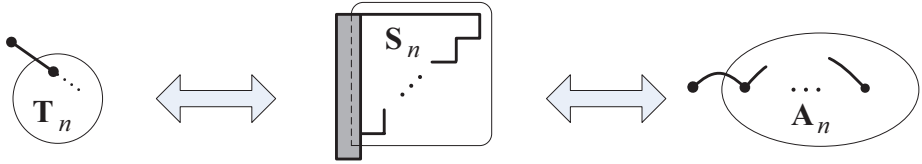}
   \caption{Construction of ${{\bf{X}}_{n + 1}}$ using ${f_R}.$}
   \label{rightchild}
\end{figure}

\item We define functions ${f_M}:{{\bf{X}}_n} \to {{\bf{X}}_{n + 1}},\,\,{\bf{X}} \in \left\{ {{\bf{T}},{\bf{S}},{\bf{A}}} \right\},\,\,n \in
\mathbb{N}$, and for the case $n=0.$
 Depending on which combinatorial object is the input to this function, we perform the following
 procedures:\\

a) $ {\bf{X}} = {\bf{T}} $: In this case, the output in
${{\bf{T}}_{n + 1}} = {f_M}\left( {{{\bf{T}}_n}} \right)$ is a tree
with $ n+1 $ vertices constructed by adding one vertex to
$t\in{{\bf{T}}_n}$ as the new root whose middle child is
the root of $t$. We define ${\bf{T}}_0$ to be $\{\emptyset\}$ and define the single element of ${\bf{T}}_1$ as $f_M(\emptyset).$ \\

b) $ {\bf{X}} = {\bf{S}} $: In this case, the output in $
{{\bf{S}}_{n + 1}} = {f_M}\left( {{{\bf{S}}_n}} \right) $
 is a
tiling with $ n+1 $ rectangles constructed by removing the left edge of
$s\in{{\bf{S}}_n}$, extending $s$ one column to the left and then adding one
single square to the bottom of the new column.
We define ${\bf{S}}_0$ to be $\{\emptyset\}$ and define the single element of ${\bf{S}}_1$ as $f_M(\emptyset).$ \\

c) $ {\bf{X}} = {\bf{A}} $: In this case, the output in ${{\bf{A}}_{n + 1}} =
{f_M}\left( {{{\bf{A}}_n}} \right)$ is an arc tree with $ n+2 $ points
constructed by adding one point to the left side of $a\in{{\bf{A}}_n}$ and
connecting it to the farthest point in $a$.
We define ${\bf{A}}_0$ to consist of a single point and define the single element of ${\bf{A}}_1$ as the image of that point under $f_M.$ \\

 Figures~\ref{midchild} and~\ref{midexamp} represent the operation of ${f_M}$ whose input can be any possible shape with size of $ n $. Thus the
number of shapes of size $ n+1 $ which can be constructed by ${f_M}$ is denoted ${c_{n + 1,2}} = {c_n}$.\\

\begin{figure}[!ht] 
   \centering
   \includegraphics[scale=.5]{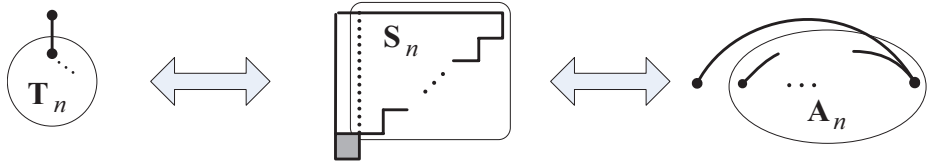}
   \caption{Construction of ${{\bf{X}}_{n + 1}}$ using ${f_M}.$}
   \label{midchild}
\end{figure}

\begin{figure}[!ht] 
   \centering
   \includegraphics[scale=.4]{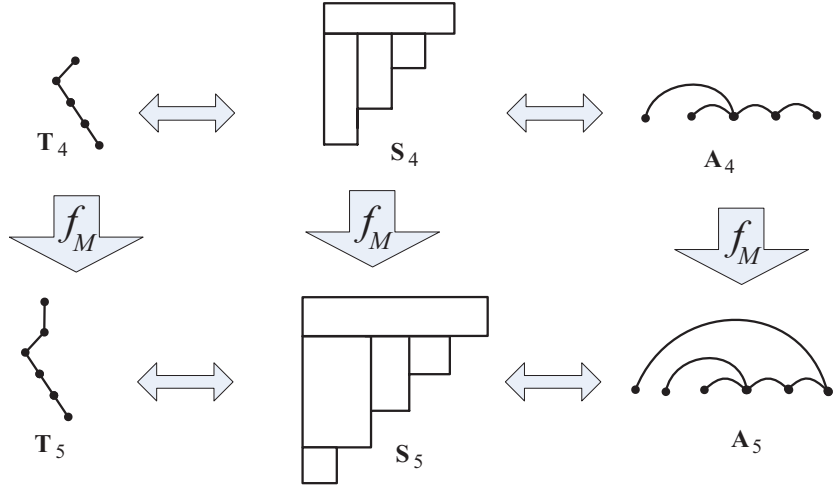}
   \caption{Examples of the construction of $ {{\bf{X}}_{n + 1}}$ using ${f_M}.$}
   \label{midexamp}
\end{figure}

\item We define functions
 ${f_L}:\left({{\bf{X}}_n}-f_M({\bf{X}}_{n - 1})\right) \to {{\bf{X}}_{n + 1}},\,\,{\bf{X}} \in \left\{ {{\bf{T},{\bf{S}},{\bf{A}}}} \right\},\,\,n \ge 1 \in \mathbb{N}$.
 Depending on which combinatorial object is the input to this function, we perform the following
 procedures:\\

a) $ {\bf{X}} = {\bf{T}} $: In this case, the output of ${f_L}$ is a
tree with $ n+1 $ vertices constructed by adding one vertex to a
size $n$ right-swept tree $t$ as the new root whose left child is
the root of
$t$. Here the original root of $t$ will not have a middle child.\\

b) $ {\bf{X}} = {\bf{S}} $: In this case, the output of ${f_L}$
 is a tiling
with $ n+1 $ rectangles constructed by adding one $\left( {(n + 1) \times 1}
\right)$ rectangle to the top of a shape from ${{\bf{S}}_n}$.
 Here
$s\in{{\bf{S}}_n}$ should not have a single square as its lowest tile. In other
words $s$ should not be constructed by ${f_M}\left(
{{{\bf{S}}_{n - 1}}} \right)$.\\

c)  $ {\bf{X}} = {\bf{A}} $: In this case, the output of ${f_L}$ is an arc tree
with $ n+2 $ points constructed by adding one point to the left side of a size
$n$ arc tree $a$ and connecting it to the second nearest point (but not the
rightmost one) in $a$ such that the connection does not intersect with any
other arc in $a$. Notice that this is impossible if the input arc tree has an
arc between its first and last points. In other words $a$ should not be
constructed by ${f_M}\left({{{\bf{A}}_{n - 1}}} \right)$.\\
\\

Figure~\ref{leftchild}  represents the operation of ${f_L}$ whose input can be
any possible shape with size of $ n $ except the shapes constructed by
${f_M}\left( {{{\bf{X}}_{n - 1}}} \right)$. Thus the number of shapes of size $
n+1 $ which can be constructed by ${f_L}$ is denoted ${c_{n + 1,3}} = {c_n}-{c_{n-1}}$.\\

\begin{figure}[!ht] 
   \centering
   \includegraphics[scale=.5]{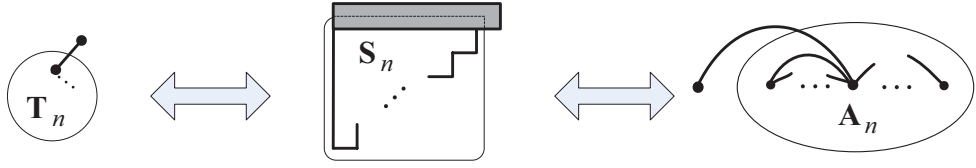}
   \caption{Construction of ${{\bf{X}}_{n + 1}}$ using ${f_L}.$}
   \label{leftchild}
\end{figure}

\item For the last case we define functions $f:\left( ({{{\mathbf{X}}_{{n_1}}}-f_M({\mathbf{X}}_{n_1-1}))\times{{\bf{X}}_{{n_2}}}} \right)
\to {{\bf{X}}_{n + 1}},\,\,\,\,\,\,{\bf{X}} \in \left\{
{{\bf{T}},{\bf{S}},{\bf{A}}} \right\},\,\,{n_1} > 1 \in
\mathbb{N},{n_2} \in\mathbb{N} ,\,\,n = ({n_1} + {n_2}) \ge 3.$
 Depending on which combinatorial object is the input to this function, we perform the following
 procedures:\\

a) $ {\bf{X}} = {\bf{T}} $: In this case, the output in
${{\bf{T}}_{{n_1} + {n_2} + 1}}$ is a tree with ${n_1} + {n_2} + 1$
vertices constructed by adding one vertex as the root whose left and
right children are the roots of trees from
${{\bf{T}}_{{n_1}}}-f_M({\bf{T}}_{n_1-1})$ and ${{\bf{T}}_{{n_2}}}$
 respectively.\\

b) $ {\bf{X}} = {\bf{S}} $: In this case, the output in $ {{\bf{S}}_{{n_1} +
{n_2} + 1}}$ is a shape with ${n_1} + {n_2} + 1$ rectangles constructed by
introducing a rectangle of size $\left( {({n_1} + 1) \times ({n_2} + 1)}
\right)$ in the top left corner of the shape and adding staircase tilings
$t_1\in{{\bf{S}}_{{n_1}}}-f_M({\bf{S}}_{n_1-1})$ and $t_2\in{{\bf{S}}_{{n_2}}}$
to the bottom and right of the rectangle
respectively. The tiling added to the bottom should not have a single square as its bottom-most tile.\\

c)$ {\bf{X}} = {\bf{A}} $: In this case, the output in $
{{\bf{A}}_{{n_1} + {n_2} + 1}}$ is a shape with points constructed
by concatenating arc trees from
${{\bf{A}}_{{n_1}}}-f_M({\bf{A}}_{n_1-1})$ and ${{\bf{A}}_{{n_2}}}$
in that order, left to right, by identifying their respective
rightmost and leftmost points. Then we add a point to the left side
of both
and connect it to the identified common point. The input arc tree on the left should not have an arc between its first and last points.\\

 Figure~\ref{forkchild} represents the operation of $f\left( {.,.}
\right)$ to construct a shape of size ${n_1} + {n_2} + 1$. The first input
argument can be any shape with size of ${n_1} > 1$ except the shapes
constructed by ${f_M}\left( {{{\bf{X}}_{{n_1} - 1}}} \right)$. However the
second argument can be any shape of size ${n_2}$ . The number of shapes of size
$ n+1 $ which can be constructed with this function is denoted ${c_{n + 1,4}} =
\sum\limits_{k = 2}^{n - 1} {{c_{n - k}}\left( {{c_k} - {c_{k - 1}}} \right)}
$.

\begin{figure}[!ht] 
   \centering
   \includegraphics[width=\textwidth]{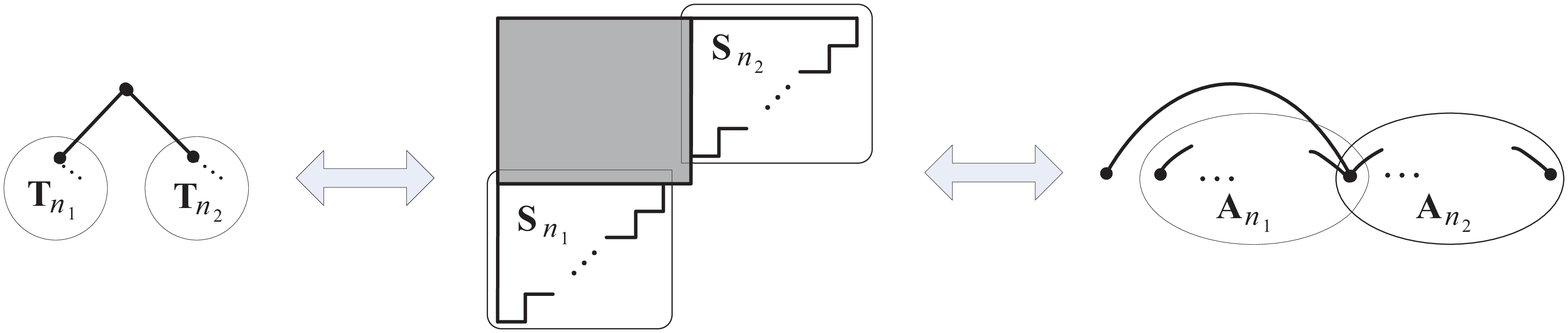}
   \caption{Construction of ${{\bf{X}}_{{n_1} + {n_2} + 1}}$ using $f\left( {{{\bf{X}}_{{n_1}}},{{\bf{X}}_{{n_2}}}} \right)$}
   \label{forkchild}
\end{figure}

\end{enumerate}

\subsection{Bijections implied by the construction.}

The functions we have defined allow the shapes to be built
recursively, and to be deconstructed as well. Unique construction
and deconstruction allow us to realize a bijection between any two
sets whose shapes are built with the four functions defined above.
First we note that there is only one element of $\mathbf{X}_1$ for
each of the shapes we consider. Figure~\ref{base} shows the three
sets of size one.

\begin{figure}
   \centering
   \includegraphics[scale=.5]{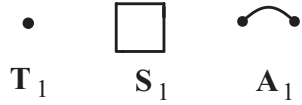}
   \caption{Trivial bijections between $\mathbf{T}_1$, $\mathbf{S}_1$ and $\mathbf{A}_1$. Recall that these are each defined as an image of $f_M.$}
   \label{base}
\end{figure}

\begin{defi}
For $n\ge 1$ we define maps $\alpha:\mathbf{X}_n\to\mathbf{X'}_n$ for
$\mathbf{X,X'}\in\{\mathbf{T,S,A}\}$ as follows:

 For $x\in X_n$ we consider $x$
to be the shape that results from applying exactly $n$ functions $f_i, i=1\dots
n$  in a particular order to $k$ initial copies of $x_0\in X_0$, the single
element of size zero for $k\ge 1$. Here $f_i \in \{f, f_L, f_R, f_M\}.$ We
denote as $F_x$ the function that is the composition of cartesian products of
the $n$ functions $f_i,$ whose domain is $k$ copies of $X_0,$ and whose sole
image is $x$.

Then  $\alpha(x) = F_x(x'_0,x'_0,\dots,x'_0)$ for $k$ copies of  $x'_0\in
\mathbf{X'}_0$, the element of size zero.
\end{defi}

For examples see Figures~\ref{sizetwo}, \ref{reduce} and~\ref{rebuild}.

\begin{theorem}
$\alpha:\mathbf{X}_n\to\mathbf{X'}_n$ as just defined gives bijections for all
$\mathbf{X,X'}\in\{\mathbf{T,S,A}\}.$
\end{theorem}
\begin{proof}
We show that $\alpha$ is well defined, surjective and invertible by
demonstrating that for any shape $x'\in \mathbf{X'}_n$ there is a unique
composition of cartesian products of functions from $f_L, f_M, f_R$ and $f$
that constructs it. Since a given composition constructs only one shape in each
of $\mathbf{T,S,A},$ having that composition means having knowledge of a unique
shape $x\in\mathbf{X}_n$ corresponding to $x'.$ The existence of a unique
composition is argued using strong induction, since the function $f$ takes
inputs from sets with smaller indices than just $n-1.$ We note that the single
shapes for $ n=1 $ (in Figure~\ref{base}) are all constructed uniquely by $f_M$
by definition. Assuming that shapes smaller than size $n$ are uniquely
constructed, we then check for size $n$ as follows:
\begin{enumerate}
\item[$\mathbf{T}:$] For any right-swept tree  $t\in {{\bf{T}}_{n}}$, depending
on whether the root has left, middle, right or both left and right children,
the tree is uniquely constructed from one or two smaller trees. For the
right-swept trees this follows from their definition.
\item[$\mathbf{S}:$] For any staircase tiling $s\in{{\bf{S}}_{n}}$ the shape is uniquely
constructed from one or two smaller shapes. The construction is determined
first by whether $s$ has a single square as its bottom-most tile. If that is
the case, then $s$ is constructed from a single smaller tiling by $f_M.$
Otherwise, we can determine whether it was constructed by $f_L, f_R$ or $f$
respectively by whether $s$ has a a single long rectangle along its top, along
its left side, or neither (instead it has a thick rectangle that covers some of
both but neither the entire top nor the entire left edges.)
\item[$\mathbf{A}:$] For any arc tree $a\in{{\bf{A}}_{n}}$ the shape is constructed from one or two smaller shapes.
The construction is determined first by whether $a$ has a single arc connecting
its first and last points. If that is the case, then $a$ is constructed from a
single smaller arc tree by $f_M.$ Otherwise, we can determine whether it was
constructed by $f_R, f_L$ or $f$ respectively by whether $a$ has a a single
short arc connecting its first (leftmost) and second points, a single arc
connecting its left-most point with the second available point, or neither
(instead it has a single longer arc connecting its left-most point to another,
more central, point.)
\end{enumerate}
\end{proof}

It is instructive to show that the total number of shapes constructed by our
four functions is equal to $ {c_{n + 1}} $. That is, that the Catalan number
${c_{n + 1}} = \sum\limits_{i = 1}^4 {{c_{n + 1,i}}} $. Equivalently we need to
prove that $\sum\limits_{k = 1}^{n - 1} {{c_k}{c_{n - k}}}  = {c_{n + 1}} -
2{c_n}$. To prove this we use Segner's recurrence relation for Catalan numbers.
\begin{equation}
{c_0} = 1\,\,\,and\,\,{c_{n + 1}} = \sum\limits_{k = 0}^n {{c_k}{c_{n - k}}} \,\,\,\,\,for\,\,\,n > 0
\end{equation}
So we have
\begin{equation}
{c_{n + 1}} = \sum\limits_{k = 0}^n {{c_k}{c_{n - k}}}  = {c_0}{c_n} + \sum\limits_{k = 1}^{n - 1} {{c_k}{c_{n - k}}}  + {c_n}{c_0} \Rightarrow \sum\limits_{k = 1}^{n - 1} {{c_k}{c_{n - k}}}  = {c_{n + 1}} - 2{c_n}
\end{equation}

\subsection{Examples}
Example 1: We want to demonstrate the bijections between right-swept trees,
staircase tilings  and arc trees for $ n=3 $.  For this we use the proposed
method twice. So for $ n=2 $ the bijection between these combinatorial objects
is illustrated in Figure~\ref{sizetwo}:

\begin{figure}[!ht]
   \centering
   \includegraphics[scale=.4]{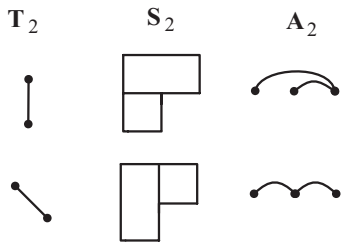}
   \caption{Bijections between right-swept
trees, staircase tilings  and arc trees for $n=2$. The top row is formed by
$f_M \circ f_M $ and the bottom row by $f_R\circ f_M.$}
   \label{sizetwo}
\end{figure}

Now the bijections between right-swept trees, staircase tilings and arc trees
for $ n=3 $ can be illustrated, in Figure~\ref{table_alpha}:

Example 2: Here we take a shape in $\mathbf{S}_{12},$ seen in
Figure~\ref{Fig9}. We want to find its images under $\alpha$ in
$\mathbf{T}_{12}$ and $\mathbf{A}_{12}$.
\begin{figure}[!ht]
   \centering
   \includegraphics[scale=.4]{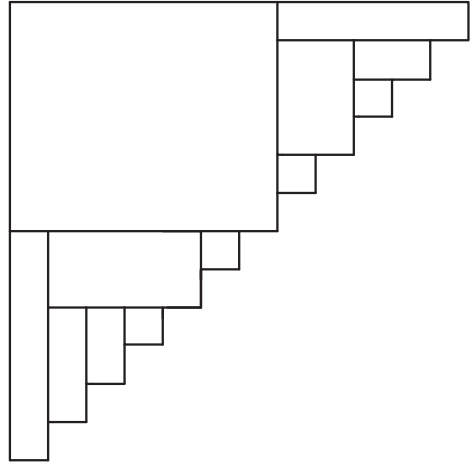}
   \caption{An example of a staircase tiling with $n=12.$}
   \label{Fig9}
\end{figure}
 First we apply the inverses
of our functions introduced before in order to uniquely reduce the size of the
shape to $n = 1$, shown in Figure~\ref{reduce}.
\begin{figure}[!ht]
   \centering
   \includegraphics[width=\textwidth]{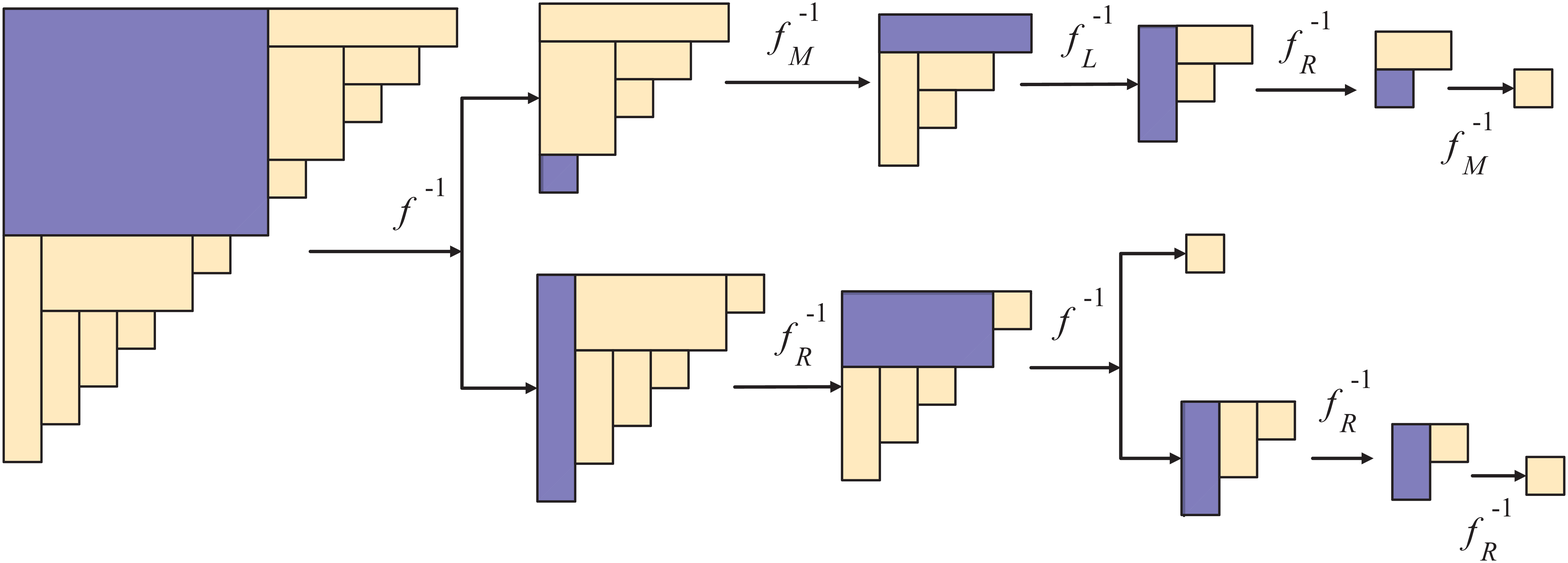}
   \caption{The inverse process on shape of Figure~\ref{Fig9}. Darkly shaded tiles are discarded by the inverse functions.
   Thus the tiling $x$ shown here is formed by:
    $${\hspace{-.45in}F_x(x_0,x_0,x_0) =f(f_M \circ f_L \circ f_R \circ f_M \circ f_M(x_0), f_R \circ f(f_M(x_0), f_R\circ f_R \circ f_M(x_0))).}$$}
   \label{reduce}
\end{figure}
Now by applying the functions we found in Figure~\ref{reduce}, we can construct
bijective images of the staircase tiling in the sets of right-swept trees and
arc trees, which are shown in Figure~\ref{rebuild}.
\begin{figure}[!ht]
   \centering
   \includegraphics[width=\textwidth]{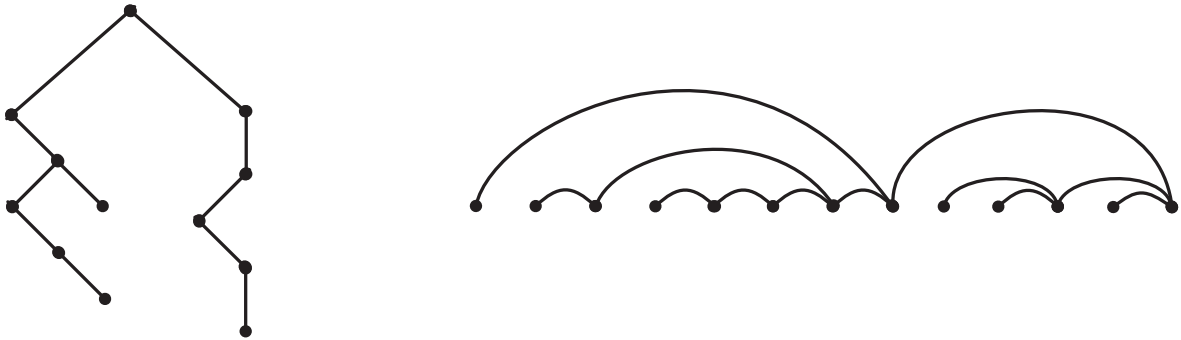}
   \caption{Images (under $\alpha$) of the shape of Figure~\ref{Fig9} in $\mathbf{T}_{12}$ and $\mathbf{A}_{12}$ }
   \label{rebuild}
\end{figure}
Finally we present the induced bijective correspondence for a binary tree with
12 internal nodes and a planar tree with 12 edges. This is seen in
Figure~\ref{twotrees}.
\begin{figure}[!ht]
   \centering
   \includegraphics[scale=.4]{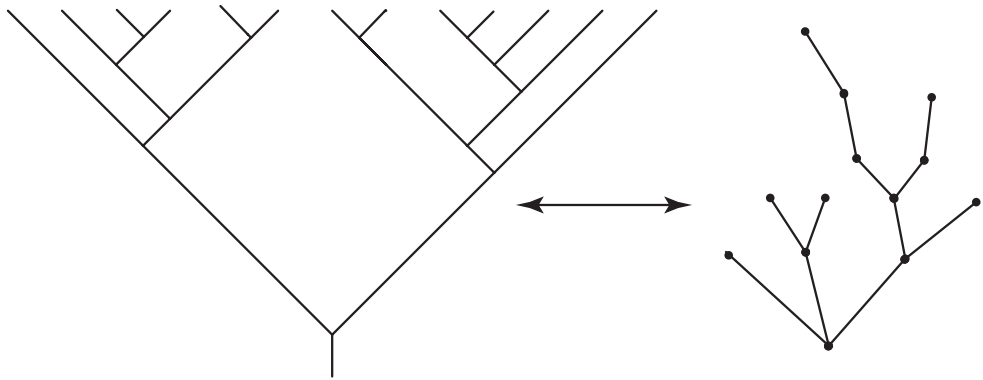}
   \caption{Corresponding binary tree and planar tree, under bijection induced by $\alpha$ between arc trees and
   staircase tilings. This example uses the staircase tiling from Figure~\ref{Fig9} and Figure~\ref{reduce}.}
   \label{twotrees}
\end{figure}

\begin{figure}[!ht]
   \centering
   \includegraphics[scale=.4]{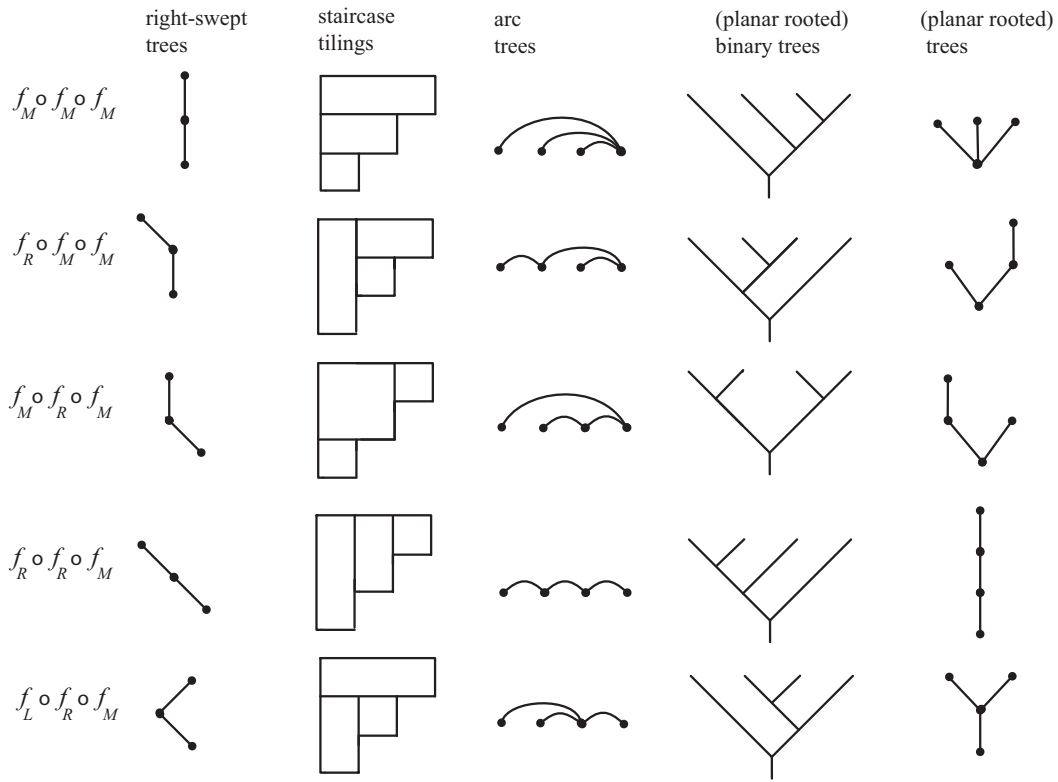}
   \caption{Bijections for $n=3.$ Each row represents a class of objects mapped to each other, the first three columns by the
     bijection $\alpha$ and the last two columns via canonical bijections from the staircase tilings and arc trees.}
   \label{table_alpha}
\end{figure}
\newpage

\section{Relative bijection from ${\mathbf{T}}_n$ to ${\mathbf{S}}_n$}\label{sec:rel}
For contrast, we consider a different method for constructing a
bijection from right-swept trees to staircase tilings.

We start by describing a second new mapping $\beta:{\mathbf{T}}_n
\to {\mathbf{S}}_n.$  Rather than using recursion, this time we
declare several rules about the relative positions of rectangles on
one hand and tree nodes on the other.  To characterize this mapping,
we need several rules which describe how two labeled nodes attached
by an edge of the right-swept tree are translated to two labeled
rectangles in the staircase tiling.

\begin{enumerate}
\item
Let two nodes be attached by a tree edge with a positive slope, so
that node $a$ is a left child of node $b.$  Then rectangle $b$ will
be immediately to the right of rectangle $a.$
 See Figure \ref{Ex1}.

\begin{figure}[!ht] 
   \centering
   \includegraphics[scale=.35]{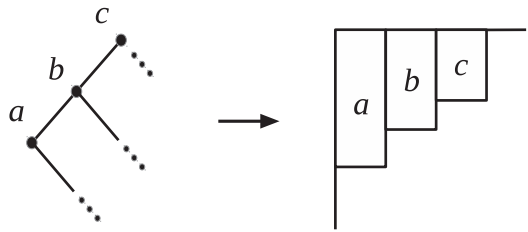}
   \caption{Example of Rule 1 ($\beta:{\mathbf{T}}_n\to{\mathbf{S}}_n$).}
   \label{Ex1}
\end{figure}

\item For two nodes attached by a negative sloped edge the situation is more complex.
If a right child $b$ is the only child of a root, middle child, or right child
$a$, and $b$ itself is a leaf or has only a middle or right child,
then the rectangle $b$ will be immediately right of the rectangle
corresponding to $a$. See Figure \ref{Ex2}.

\begin{figure}[!ht] 
   \centering
   \includegraphics[scale=.35]{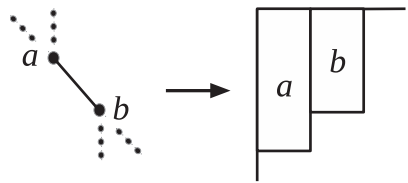}
   \caption{Example of Rule 2 ($\beta:{\mathbf{T}}_n\to{\mathbf{S}}_n$).}
   \label{Ex2}
\end{figure}

\item However, if a right child is produced from a left child (or as part of a left and right child),
 the corresponding rectangle will be immediately below the rectangle corresponding to the spawning vertex.  See Figure \ref{Ex3}.

\begin{figure}[!ht] 
   \centering
   \includegraphics[scale=.35]{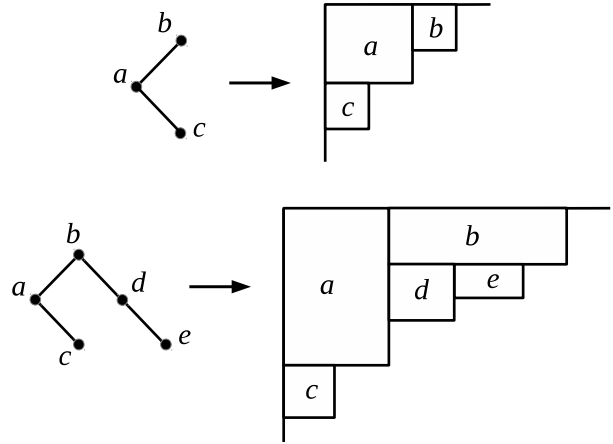}
   \caption{Examples of Rule 3 ($\beta:{\mathbf{T}}_n\to{\mathbf{S}}_n$).}
   \label{Ex3}
\end{figure}

\item A middle child $b$ will always correspond to a rectangle directly below the rectangle corresponding to the spawning vertex $a$.  See Figure \ref{Ex4}.

\begin{figure}[!ht] 
   \centering
   \includegraphics[scale=.35]{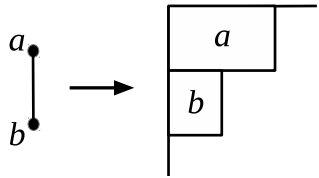}
   \caption{Example of Rule 4 ($\beta:{\mathbf{T}}_n\to{\mathbf{S}}_n$).}
   \label{Ex4}
\end{figure}

\item There is only one case in which adjacent nodes do not correspond to adjacent rectangles: if  $b$ is a right child of $a$,
and $b$ has a left child $d$. Then rectangle $b$ is right of
rectangle $a$, but Rule 1 is used to place the left child of $b$
(and its left child, etc.) between rectangles $a$ and $b$. See
Figure \ref{Ex5}.

\begin{figure}[!ht] 
   \centering
   \includegraphics[scale=.35]{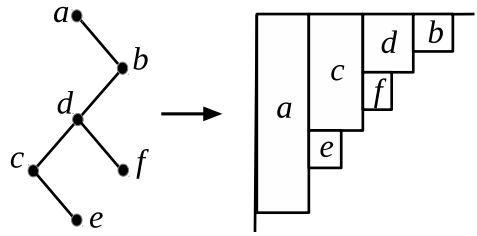}
   \caption{Example of Rule 5 ($\beta:{\mathbf{T}}_n\to{\mathbf{S}}_n$).}
   \label{Ex5}
\end{figure}

\end{enumerate}

This method gives a specific set of instructions at each vertex
point for how to proceed with no ambiguity in the decision-making
process.  Therefore each tree in ${\mathbf{T}}_n$ will give a unique
structure in ${\mathbf{S}}_n$.

Figure \ref{Example} is a final example for the case that $n = 10$
for the mapping from ${\mathbf{T}}_n$ to ${\mathbf{S}}_n$.

\begin{figure}[!ht] 
   \centering
   \includegraphics[scale=.35]{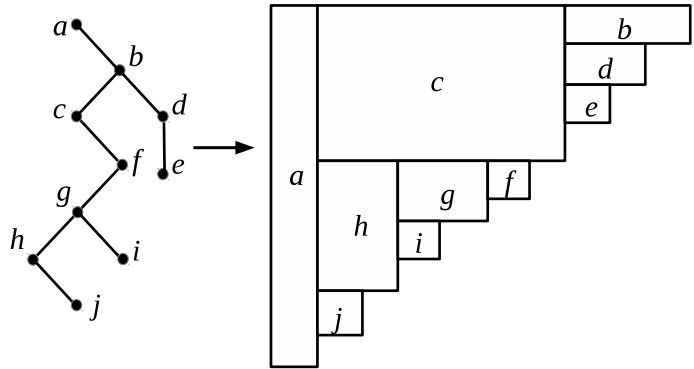}
   \caption{A full example for the case $n = 10$ ($\beta:{\mathbf{T}}_n\to{\mathbf{S}}_n$).}
   \label{Example}
\end{figure}

\begin{theorem}
The mapping $\beta:{\mathbf{T}}_n\to{\mathbf{S}}_n$ determined by the above
rules is a bijection.
\end{theorem}
\begin{proof}
We consider the reverse mapping $\beta^{-1}:{\mathbf{S}}_n\to{\mathbf{T}}_n$.
In a similar fashion, we develop a series of rules for this mapping.

\begin{enumerate}

\item If the top-left rectangle goes to the bottom of the figure (i.e. width = $1$ unit),
 the root spawns a right child.  If the top-left rectangle goes to the farthest right edge (i.e. depth = $1$ unit),
the root spawns a middle child.  This process is repeated as necessary.  See Figure \ref{Exmpl1}.

\begin{figure}[!ht] 
   \centering
   \includegraphics[scale=.35]{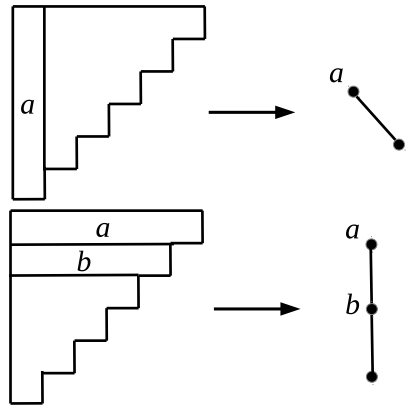}
   \caption{Examples of Rule 1 ($\beta^{-1}:{\mathbf{S}}_n\to{\mathbf{T}}_n$).}
   \label{Exmpl1}
\end{figure}

\item If the top-left rectangle has width greater than 1 unit and
depth greater than 1 unit, then a limb of left children is formed
where
 the bottom vertex on this limb corresponds to the left-most rectangle.
 The length of this limb of left children is determined by the number of
 rectangles read from left to right, going as far right as possible.  See Figure \ref{Exmpl2}.

\begin{figure}[!ht] 
   \centering
   \includegraphics[scale=.35]{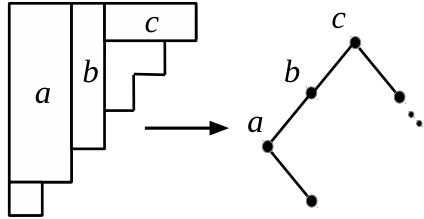}
   \caption{Example of Rule 2 ($\beta^{-1}:{\mathbf{S}}_n\to{\mathbf{T}}_n$).}
   \label{Exmpl2}
\end{figure}

\item Any remaining rectangles are treated as right children of the vertex corresponding to the
rectangle directly above and the process repeats with these remaining rectangles acting as miniature
 versions of ${\mathbf{S}}_n$.  Notice this rule satisfies the restriction in ${\mathbf{T}}_n$ to have
  a right child or a right and left child following a left child.  See Figure \ref{Exmpl3}.

\begin{figure}[!ht] 
   \centering
   \includegraphics[scale=.35]{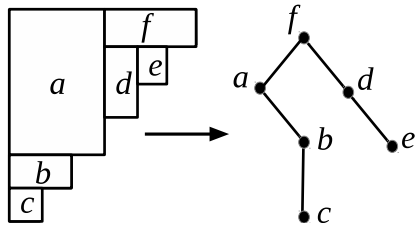}
   \caption{Example of Rule 3 $(\beta^{-1}:{\mathbf{S}}_n\to{\mathbf{T}}_n$).}
   \label{Exmpl3}
\end{figure}

\end{enumerate}

As in the previous case, we now illustrate with a final example, the
reverse image for the case $n=10$ that we considered earlier.  See
Figure \ref{Example1}.  We use the set of rules developed here, and
see that we arrive at the same pre-image in ${\mathbf{T}}_n$.

\begin{figure}[!ht] 
   \centering
   \includegraphics[scale=.35]{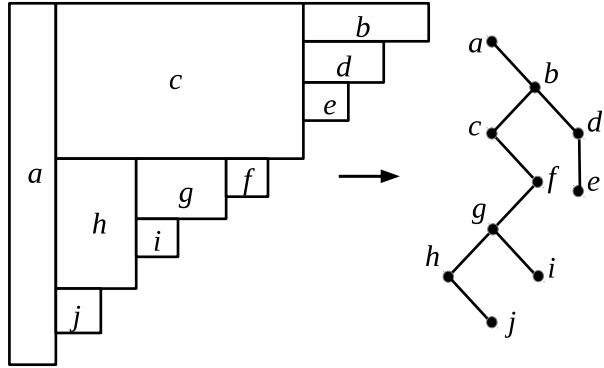}
   \caption{A full example for the case $n = 10$ ($\beta^{-1}:{\mathbf{S}}_n\to{\mathbf{T}}_n$).}
   \label{Example1}
\end{figure}

Once again, no ambiguity arises from the rules developed above, and so each
output of this algorithm is unique for each unique input.

We now argue that these sets of rules form an inverse function.  We define Tier
One rules to be Rules 2 and 4 from the former direction and Rule 1 from the
latter direction.  We define Tier Two rules to be Rules 1 and 5 in the former
direction and Rule 2 in the latter direction.  We define Tier Three rules to be
Rule 3 from the former direction and Rule 3 from the latter direction.

Tier One rules are easily seen to be inverse rules, and when using
any Tier One rule from the outset, the resulting figure in the next
step is a new tree or staircase shape with $n-1$ vertices or
rectangles for ${\mathbf{T}}_n$ and ${\mathbf{S}}_n$, respectively.

The real key to this process occurs in Tier Two and Tier Three
rules.  In ${\mathbf{T}}_n$, the Tier Two rules occur any time a
left child is introduced, whether it is from the root (Rule 1) or
somewhere else in the tree (Rule 5).  In ${\mathbf{S}}_n$, the Tier
Two rules occur any time a rectangle having width and depth both
greater than 1 unit is introduced, either as the top-left rectangle
(corresponding to the root in ${\mathbf{T}}_n$) or somewhere else in
the staircase (corresponding to a different branch in
${\mathbf{T}}_n$).  These two phases are clearly inverses of each
other, since Rules 1 and 5 of ${\mathbf{T}}_n$ imply Rule 2 of
${\mathbf{S}}_n$ and vice versa, and in corresponding sections of
the tree and staircase.

Tier Three rules in ${\mathbf{T}}_n$ occur whenever a right child
branches off of a left limb (where the length of the limb is
anywhere from $1$ to $n-1$ branches). Similarly, Tier Three rules in
${\mathbf{S}}_n$ occur any time there are leftover rectangles
underneath of a Tier Two structure in ${\mathbf{S}}_n$.  In both
${\mathbf{T}}_n$ and ${\mathbf{S}}_n$, the process renews itself
when Tier Three rules are utilized, leaving smaller tree and
staircase structures of corresponding size, both starting
independently with the same set of rules the larger structure obeys.
Therefore, Tier Three rules in ${\mathbf{T}}_n$ imply Tier Three
rules in ${\mathbf{S}}_n$ and vice verse, and in corresponding
sections of the tree and staircase.

Breaking down our algorithm into three tiers of rules has allowed us
to show that this function is indeed an inverse function.  We
therefore have successfully described the bijection between
${\mathbf{S}}_n$ and ${\mathbf{T}}_n$.
\end{proof}

\subsection{Examples contrasting the bijections.}
Interestingly, the two bijections $\alpha:{\mathbf{T}}_{n} \to
{\mathbf{S}}_{n}$ and $\beta:{\mathbf{T}}_{n}\to {\mathbf{S}}_{n}$ set up
precisely the same correspondence between right-swept trees and staircase
tilings for $n=0,1, 2.$  An obvious question is raised: are the two bijections
we have described the same? The answer is no. We see this at $n=3,$ by
comparing the tables in Figures~\ref{table_beta} and~\ref{table_alpha}.

\begin{figure}[!ht]
   \centering
   \includegraphics[scale=.4]{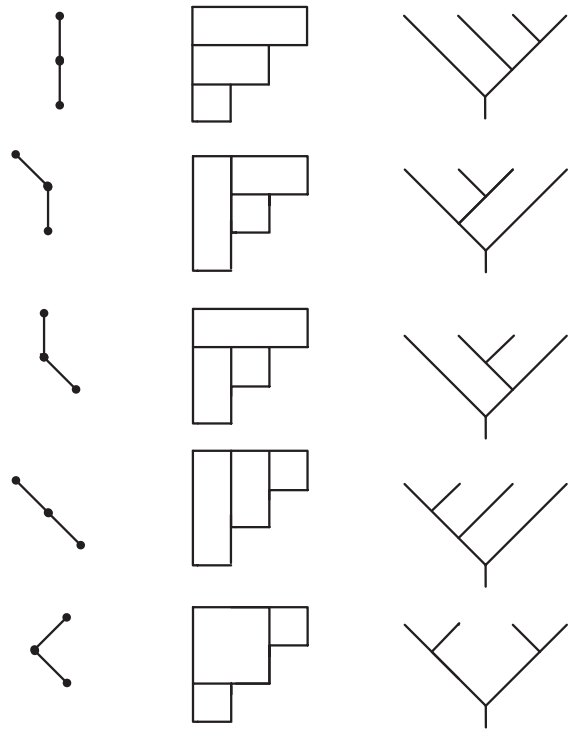}
   \caption{Bijections for $n=3$ for the
     bijection $\beta.$ Note that the third and fifth images are switched from those of $\alpha$ in Figure~\ref{table_alpha}. }
   \label{table_beta}
\end{figure}

The slightly larger example we include next in Figure~\ref{fiver} was suggested
by the referee. In fact we'd like to take this opportunity to thank the
referee, who went to a great deal of trouble on our behalf. This example is in
$n=5.$

\begin{figure}[!ht]
   \centering
   \includegraphics[scale=.3]{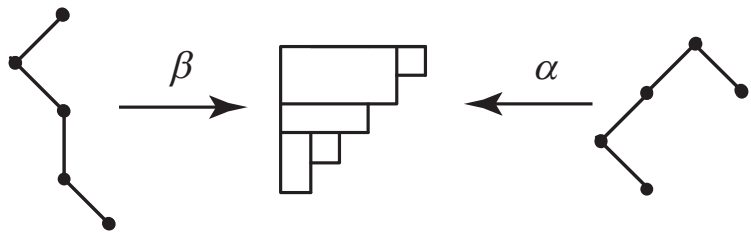}
   \caption{Contrasted pre-images of an ${\mathbf{S}}_5$ tiling. }
   \label{fiver}
\end{figure}

 Finally we include here
in Figure~\ref{f:compare} a larger example to highlight the differences. A tree
from ${\mathbf{T}}_{12}$ (the same example as in Figure~\ref{rebuild}) is shown
in the center, and then its two images in ${\mathbf{S}}_{12}$: on the left is
the image of the recursive bijection from Section~\ref{sec:rec} and on the
right the image of the relative bijection from Section~\ref{sec:rel}.

\begin{figure}[!ht] 
   \centering
      \includegraphics[scale=.35]{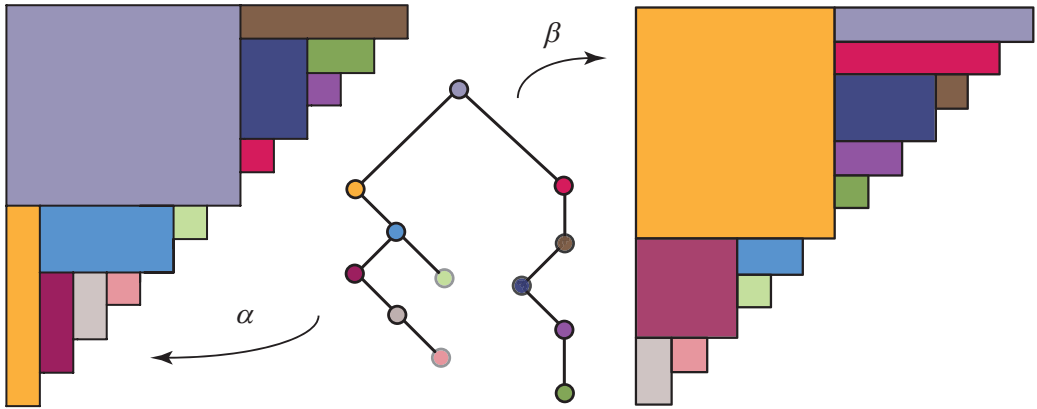}
   \caption{Two images of a right-swept tree from ${\mathbf{T}}_{12}$.}
   \label{f:compare}
\end{figure}

\bibliography{bijectbib}{}

\begin{thebibliography}{1}

\bibitem{mernik}
Matej Crepinsek and Luka Mernik.
\newblock An efficient representation for solving catalan number related
  problems.
\newblock {\em Int. J. of Pure and Applied Math.}, 56(4):589--604, 2009.

\bibitem{kim}
Jang~Soo Kim.
\newblock Front representation of set partitions.
\newblock {\em SIAM J. Discrete Math.}, 25(1):447--461, 2011.

\bibitem{reading-direc}
Shirley Law and Nathan Reading.
\newblock The hopf algebra of diagonal rectangulations.
\newblock {\em J. Combin. Theory Ser. A.}, 119(3):788--824, 2012.

\bibitem{sloane}
N.~J.~A. Sloane.
\newblock {\em The On-Line Encyclopedia of Integer Sequences}.
\newblock http://oeis.org, 2011.

\bibitem{Stan}
Richard~P. Stanley.
\newblock {\em Enumerative combinatorics. {V}ol. 2}, volume~62 of {\em
  Cambridge Studies in Advanced Mathematics}.
\newblock Cambridge University Press, Cambridge, 1999.
\newblock With a foreword by Gian-Carlo Rota and appendix 1 by Sergey Fomin.

\end{thebibliography}
\bibliographystyle{plain}

\end{document}